\documentclass[reqno]{amsart}
\usepackage{amsmath, amsthm, amssymb, amstext}

\usepackage[left=3.8cm,right=3.8cm,top=3cm,bottom=3cm,includeheadfoot]{geometry}
\usepackage{hyperref,xcolor}
\hypersetup{
pdfborder={0 0 0},
colorlinks,
}
\usepackage{enumitem}
\usepackage{todonotes}

\newtheorem{theorem}{Theorem}
\newtheorem{remark}[theorem]{Remark}

\newtheorem{proposition}[theorem]{Proposition}

\DeclareMathOperator*{\divergenz}{div}              %
\DeclareMathOperator*{\Ss}{S}

\newcommand{\N}{\mathbb{N}}

\newcommand{\R}{\mathbb{R}}

\newcommand{\Lp}[1]{L^{#1}(\Omega)}

\newcommand{\Lprand}[1]{L^{#1}(\partial\Omega)}
\newcommand{\Wp}[1]{W^{1,#1}(\Omega)}

\newcommand{\Wpzero}[1]{W^{1,#1}_0(\Omega)}

\newcommand{\lan}{\langle}
\newcommand{\ran}{\rangle}
\newcommand{\eps}{\varepsilon}
\newcommand{\ph}{\varphi}

\newcommand{\into}{\int_{\Omega}}

\newcommand{\weak}{\rightharpoonup}

\newcommand{\Linf}{L^{\infty}(\Omega)}
\newcommand{\close}{\overline{\Omega}}

\newcommand{\cprime}{$'$}

\renewcommand{\l}{\left}
\renewcommand{\r}{\right}
\numberwithin{theorem}{section}
\numberwithin{equation}{section}

\title[Sign changing solution for a double phase problem]{Sign changing solution for a double phase problem with nonlinear boundary condition via the Nehari manifold}

\author[L.\,Gasi\'nski]{Leszek Gasi\'nski}
\address[L.\,Gasi\'nski]{Pedagogical University of Cracow, Department of Mathematics, Podchorazych 2, 30-084 Cracow, Poland}
\email{leszek.gasinski@up.krakow.pl}

\author[P.\,Winkert]{Patrick Winkert}
\address[P.\,Winkert]{Technische Universit\"{a}t Berlin, Institut f\"{u}r Mathematik, Stra\ss e des 17.\,Juni 136, 10623 Berlin, Germany}
\email{winkert@math.tu-berlin.de}


\begin{document}

\begin{abstract}
    In this paper we study quasilinear elliptic equations driven by the so-called double phase operator and with a nonlinear boundary condition. Due to the lack of regularity, we prove the existence of multiple solutions by applying the Nehari manifold method along with truncation and comparison techniques and critical point theory. In addition, we can also determine the sign of the solutions (one positive, one negative, one nodal). Moreover, as a result of independent interest, we prove for a general class of such problems the boundedness of weak solutions.
\end{abstract}
	
\maketitle
	
\section{Introduction}

Given a bounded domain $\Omega \subseteq \R^N$, $N \geq 2$, with Lipschitz boundary $\partial \Omega$, we study the following double phase problem with nonlinear boundary condition
\begin{equation}\label{problem}
    \begin{aligned}
	-\divergenz\left(|\nabla u|^{p-2}\nabla u+\mu(x) |\nabla u|^{q-2}\nabla u\right) & =f(x,u)-|u|^{p-2}u-\mu(x)|u|^{q-2}u\quad && \text{in } \Omega,\\
	\left(|\nabla u|^{p-2}\nabla u+\mu(x) |\nabla u|^{q-2}\nabla u\right) \cdot \nu & = g(x,u) &&\text{on } \partial \Omega,
    \end{aligned}
\end{equation}
where $\nu(x)$ denotes the outer unit normal of $\Omega$ at $x\in \partial \Omega$, $1<p<q<N$, $\mu \in\Linf$ such that $\mu(x) \geq 0$ for almost all (a.\,a.) $x\in\Omega$ and $f\colon\Omega\times\R\to \R$, $g\colon\partial\Omega\times\R\to\R$ are Carath\'{e}odory functions which have $(p-1)$-superlinear growth near $\pm\infty$.

The differential operator in \eqref{problem} is the so-called double phase operator and is given by
\begin{align}\label{operator_double_phase}
    -\divergenz \l(|\nabla u|^{p-2} \nabla u+ \mu(x) |\nabla u|^{q-2} \nabla u\r)\quad \text{for }u\in \Wp{\mathcal{H}}
\end{align}
with an appropriate Musielak-Orlicz Sobolev space $\Wp{\mathcal{H}}$, see its definition in Section \ref{section_2}.
Special cases of \eqref{operator_double_phase}, studied extensively in the literature, occur when $\inf_{\close} \mu>0$ (the weighted $(q,p)$-Laplacian) or when $\mu\equiv 0$ (the $p$-Laplace differential operator).
The operator \eqref{operator_double_phase} is related to the energy functional
\begin{align}\label{integral_minimizer}
   u \mapsto \int_\Omega \big(|\nabla  u|^p+\mu(x)|\nabla  u|^q\big)\,dx,
\end{align}
where the integrand $H(x,\xi)=|\xi|^p+\mu(x)|\xi|^q$ for all $(x,\xi) \in \Omega\times \R^N$ has unbalanced growth, that is,
\begin{align*}
    |\xi|^p \leq H(x,\xi) \leq b \l(1+|\xi|^q\r)\quad \text{for a.\,a.\,}x\in\Omega \text{ and for all }\xi\in\R^N,
\end{align*}
with $b>0$. The integral functional \eqref{integral_minimizer} is characterized by the fact that the energy density changes its ellipticity and growth properties according to the point in the domain. More precisely, its behavior depends on the values of the weight function $\mu(\cdot)$. Indeed, on the set $\{x\in \Omega: \mu(x)=0\}$ it will be controlled by the gradient of order $p$ and in the case $\{x\in \Omega: \mu(x) \neq 0\}$ it is the gradient of order $q$. This is the reason why it is called double phase.

Originally, Zhikov \cite{Zhikov-1986} was the first who studied so-called double phase functionals of the form \eqref{integral_minimizer}  in order to describe models of strongly anisotropic materials, see also Zhikov \cite{Zhikov-1995}, \cite{Zhikov-1997} and the monograph of Zhikov-Kozlov-Oleinik \cite{Zhikov-Kozlov-Oleinik-1994}.
Functionals like \eqref{integral_minimizer} have been studied by several authors with respect to regularity and nonstandard growth. We refer to the works of Baroni-Colombo-Mingione \cite{Baroni-Colombo-Mingione-2015}, \cite{Baroni-Colombo-Mingione-2016}, \cite{Baroni-Colombo-Mingione-2018}, Baroni-Kuusi-Mingione \cite{Baroni-Kuusi-Mingione-2015}, Cupini-Marcellini-Mascolo \cite{Cupini-Marcellini-Mascolo-2015}, Co\-lom\-bo-Mingione \cite{Colombo-Mingione-2015a}, \cite{Colombo-Mingione-2015b},
 Marcellini \cite{Marcellini-1989}, \cite{Marcellini-1991}, Ok \cite{Ok-2018}, \cite{Ok-2020}, Ragusa-Tachikawa \cite{Ragusa-Tachikawa-2020} and the references therein. We also mention the recent works of Beck-Mingione \cite{Beck-Mingione-2020}, \cite{Beck-Mingione-2019} concerning nonuniformly elliptic variational problems.

In general, double phase differential operators and corresponding energy functionals given in \eqref{operator_double_phase} and \eqref{integral_minimizer}, respectively, appear in several physical applications. For example, in the elasticity theory, the modulating coefficient $\mu(\cdot)$ dictates the geometry of composites made of two different materials with distinct power hardening exponents $q$ and $p$, see Zhikov \cite{Zhikov-2011}.
We also refer to other applications which can be found in the works of
Bahrouni-R\u{a}dulescu-Repov\v{s} \cite{Bahrouni-Radulescu-Repovs-2019} on transonic flows, Benci-D'Avenia-Fortunato-Pisani \cite{Benci-DAvenia-Fortunato-Pisani-2000} on quantum physics and  Cherfils-Il\cprime yasov \cite{Cherfils-Ilyasov-2005} on reaction diffusion systems.

The aim of our paper is to prove multiplicity results for problems of the form \eqref{problem} where the nonlinearities are supposed to be $(p-1)$-superlinear at $\pm\infty$. Due to the lack of regularity for problems \eqref{problem}, several tools, which are usually applied in the theory of multiplicity results based on the regularity results of Lieberman \cite{Lieberman-1988} and Pucci-Serrin \cite{Pucci-Serrin-2007}, cannot be used in our treatment. Instead we will make use of the so-called Nehari manifold which was first introduced by Nehari in the works \cite{Nehari-1961}, \cite{Nehari-1960}. This method developed into a very powerful tool in order to find solutions (especially, sign-changing solutions) via critical point theory. The idea in this method is the following: Let $X$ be a real Banach space and let $J \in C^1(X,\R)$ be a functional. If $u \neq 0$ is a critical point of $J$, then $u$ belongs to the set
\begin{align*}
    \mathcal{N} =\Big\{u \in X\setminus\{0\} \,:\, \lan J'(u),u\ran=0 \Big\},
\end{align*}
where $\lan\cdot,\cdot\ran$ is the duality paring between $X$ and its dual space $X^*$. Therefore, $\mathcal{N}$ is an appropriate constraint for finding nontrivial critical points of $J$. Although $\mathcal{N}$ may not be a manifold in general, it is called Nehari manifold. So, we are looking for nontrivial minimizers of the functional $J$ in a subset of the whole space which contains the nontrivial critical points of $J$, namely $\mathcal{N}$.  We refer to the book chapter of Szulkin-Weth \cite{Szulkin-Weth-2010} which describes the method very well. Although there is no regularity theory for double phase problems, we are also going to prove a boundedness result for weak solutions of \eqref{problem} by using Moser's iteration which can be seen as a starting point in order to obtain the smoothness of the solutions.

A pioneer work for multiplicity results with superlinear nonlinearities was published by Wang \cite{Wang-1991} for semilinear Dirichlet problems driven by the Laplacian. Although double phase problems have been known for a while, existence results have only been obtained by few authors. Perera-Squassina \cite{Perera-Squassina-2018} showed the existence of a solution of problem \eqref{problem} with Dirichlet boundary condition by applying Morse theory where they used a cohomological local splitting to get an estimate of the critical groups at zero. The corresponding eigenvalue problem of the double phase operator with Dirichlet boundary condition has been studied by Colasuonno-Squassina \cite{Colasuonno-Squassina-2016} who proved the existence and properties of related variational eigenvalues. By applying variational methods, Liu-Dai \cite{Liu-Dai-2018} treated double phase problems and proved existence and multiplicity results, as well as sign-changing solutions. A similar treatment has been recently done by Gasi\'nski-Papageorgiou \cite[Proposition 3.4]{Gasinski-Papageorgiou-2019} via the Nehari manifold method for locally Lipschitz continuous right-hand sides. Furthermore, we refer to a recent work of the authors \cite{Gasinski-Winkert-2020b} in which the existence of at least one solution for Dirichlet double phase problems with convection is shown by applying the surjectivity result for pseudomonotone operators. This can be realized by an easy condition on the convection term, in addition to the usual growth condition. For multiple constant sign solutions we refer to another work of the authors in \cite{Gasinski-Winkert-2020a}. To the best of our knowledge this is the first work dealing with a double phase phenomenon along with a nonlinear boundary condition.

Finally, we mention recent papers which are very close to our topic dealing with certain types of double phase problems. We refer to Bahrouni-R\u{a}dulescu-Winkert \cite{Bahrouni-Radulescu-Winkert-2019}, Cencelj-R\u{a}dulescu-Repov\v{s} \cite{Cencelj-Radulescu-Repovs-2018}, Marino-Winkert \cite{Marino-Winkert-2020}, Papageorgiou-R\u{a}dulescu-Repov\v{s} \cite{Papageorgiou-Radulescu-Repovs-2019b}, \cite{Papageorgiou-Radulescu-Repovs-2018}, \cite{Papageorgiou-Radulescu-Repovs-2020}, R\u{a}dulescu \cite{Radulescu-2019}, Zhang-R\u{a}dulescu \cite{Zhang-Radulescu-2018}, Zeng-Gasi\'nski-Winkert-Bai  \cite{Zeng-Gasinski-Winkert-Bai-2020}, \cite{Zeng-Bai-Gasinski-Winkert-2020} and the references therein.

The paper is organized as follows. In Section \ref{section_2} we recall the definition of the Musielak-Orlicz space $\Lp{\mathcal{H}}$ and its corresponding Musielak-Orlicz Sobolev space $\Wp{\mathcal{H}}$ and we give some embedding results dealing with boundary Lebesgue spaces following the work of Colasuonno-Squassina \cite{Colasuonno-Squassina-2016}.
In Section \ref{section_3} we present a boundedness result for a more general class of problems than \eqref{problem} following the treatment of Marino-Winkert \cite{Marino-Winkert-2019}, see Theorem \ref{theorem_apriori}. In Section \ref{section_4} we state the full assumptions on the data of problem \eqref{problem}, give the definition of the weak solution and state and prove our existence result concerning constant sign solutions, see Proposition \ref{proposition_constant_sign_solutions}. Finally, in Section \ref{section_5}, we prove the existence of a sign-changing solution by applying the Nehari manifold method described above and state our full multiplicity result, see Theorem \ref{theorem_main_result}.

\section{Preliminaries}\label{section_2}

In this section we give some definitions and results which will be needed later. We denote by $\Lp{r}$ and $L^r(\Omega;\R^N)$ the usual Lebesgue spaces endowed with the norm $\|\cdot\|_r$ for $1\leq r<\infty$ and by $\Wp{r}$ and $\Wpzero{r}$ we identify the corresponding Sobolev spaces equipped with the norms $\|\cdot\|_{1,r}$ and $\|\cdot\|_{1,r,0}$, respectively, for $1< r<\infty$. From the Sobolev embedding theorem it is clear that we have the embedding
\begin{align*}
    \Wp{r} \to \Lp{\hat{r}},
\end{align*}
which is compact for any $\hat{r}<r^*$ and continuous for $\hat{r}=r^*$, where $r^*$ is the critical exponent of $r$ defined by
\begin{align}\label{critical_exponent_domain}
    r^*=
    \begin{cases}
	\frac{Nr}{N-r} &\text{if }r<N,\\
	\text{any }\ell\in(r,\infty) & \text{if }r\geq N.
    \end{cases}
\end{align}

Furthermore, we consider the $(N-1)$-dimensional Hausdorff (surface) measure $\sigma$ on the boundary $\partial \Omega$ of $\Omega$. Based on this, we can introduce in the usual way the boundary Lebesgue space $\Lprand{r}$ with norm $\|\cdot\|_{r,\partial\Omega}$. It is well-known that there exists a unique continuous linear operator $\gamma\colon W^{1,r}(\Omega) \to \Lprand{\tilde{r}}$ with $\tilde{r}\leq r_*$, called trace map, such that
\begin{align*}
    \gamma (u)=u \big|_{\partial \Omega} \quad \text{for all }u\in W^{1,p}(\Omega) \cap C^0(\close).
\end{align*}
Here, $r_*$ is the critical exponent of $r$ on the boundary given by
\begin{align}\label{critical_exponent_boundary}
    r_*=
    \begin{cases}
	\frac{(N-1)r}{N-r} &\text{if }r<N,\\
	\text{any }\ell\in(r,\infty) & \text{if }r\geq N.
    \end{cases}
\end{align}
By the trace embedding theorem we know that $\gamma$ is compact for any $\tilde{r}<r_*$. So, we understand all restrictions of Sobolev functions to $\partial \Omega$ in the sense of traces. For simplification we will avoid the notation of the trace operator in this paper.

In the entire paper we assume that
\begin{align}\label{condition_poincare}
    1<p<q<N, \qquad  \frac{Nq}{N+q-1}<p, \qquad \mu \in\Linf,\, \mu(x) \geq 0 \text{ for a.\,a.\,} x\in\Omega.
\end{align}
\begin{remark}
	Note that the second inequality in \eqref{condition_poincare} is equivalent to the inequality $q<p_*$ and so $q<p^*$ is true as well. Hence, we have the compactness of $\Wp{\mathcal{H}}$ into suitable Lebesgue spaces defined on the domain and also on the boundary, see Proposition \ref{proposition_embeddings} below. We point out that \eqref{condition_poincare} is different from the usual conditions when dealing with Dirichlet double phase problems, see, for example, \cite{Gasinski-Winkert-2020a} and \cite{Gasinski-Winkert-2020b} of the authors. Indeed, in these papers it is supposed that 
	\begin{align}\label{condition_poincare2}
		1<p<q<N, \qquad  \frac{q}{p}<1+\frac{1}{N}, \qquad 0 \leq \mu(\cdot)\in C^{0,1}(\close).
	\end{align}
	Condition \eqref{condition_poincare2} was used for the first time by Baroni-Colombo-Mingione \cite[see (1.8)]{Baroni-Colombo-Mingione-2015} in order to obtain regularity results of local minimizers for double phase integrals, see also the related works \cite{Baroni-Colombo-Mingione-2016} and \cite{Baroni-Colombo-Mingione-2018} of the same authors and Colombo-Mingione \cite{Colombo-Mingione-2015a}, \cite{Colombo-Mingione-2015b}. The meaning of \eqref{condition_poincare2} is twofold. On the one hand, from \eqref{condition_poincare2}, we know that smooth functions are dense in the Musielak-Orlicz Sobolev space $\Wp{\mathcal{H}}$, see, for example, Colasuonno-Squassina \cite[Proposition 6.5]{Colasuonno-Squassina-2016} or Harjulehto-H\"{a}st\"{o} \cite[Theorem 6.4.7 and Section 7.2]{Harjulehto-Hasto-2019}. On the other hand, \eqref{condition_poincare2} is required to have an equivalent norm on the space $\Wpzero{\mathcal{H}}$, see Colasuonno-Squassina \cite[Proposition 2.18(iv)]{Colasuonno-Squassina-2016}. Since we do not need both arguments in our work, we suppose the conditions stated in \eqref{condition_poincare}. As far as we know there is no relationship between
	\begin{align*}
		\frac{Nq}{N+q-1}<p\quad\text{and}\quad \frac{q}{p}<1+\frac{1}{N},
	\end{align*}
	only that both inequalities imply that $q<p^*$.
\end{remark}
Now, let $\mathcal{H}\colon \Omega \times [0,\infty)\to [0,\infty)$ be the function
\begin{align*}
    (x,t)\mapsto t^p+\mu(x)t^q.
\end{align*}
We set
\begin{align}\label{modular}
    \rho_{\mathcal{H}}(u):=\into \mathcal{H}(x,|u|)\,dx=\into \big(|u|^{p}+\mu(x)|u|^q\big)\,dx.
\end{align}
Based on the definition of the modular function $\rho_{\mathcal{H}}$ we are now in the position to introduce the so-called Musielak-Orlicz space $L^\mathcal{H}(\Omega)$ which is defined by
\begin{align*}
    L^\mathcal{H}(\Omega)=\left \{u ~ \Big | ~ u\colon \Omega \to \R \text{ is measurable and } \rho_{\mathcal{H}}(u)<+\infty \right \}
\end{align*}
equipped with the Luxemburg norm
\begin{align*}
    \|u\|_{\mathcal{H}} = \inf \left \{ \tau >0 : \rho_{\mathcal{H}}\left(\frac{u}{\tau}\right) \leq 1  \right \}.
\end{align*}
From Colasuonno-Squassina \cite[Proposition 2.14]{Colasuonno-Squassina-2016} we know that the space $L^\mathcal{H}(\Omega)$ is uniformly convex and so reflexive. Furthermore, we introduce the seminormed space
\begin{align*}
    L^q_\mu(\Omega)=\left \{u ~ \Big | ~ u\colon \Omega \to \R \text{ is measurable and } \into \mu(x) |u|^q dx< +\infty \right \}
\end{align*}
endowed with the seminorm
\begin{align*}
    \|u\|_{q,\mu} = \left(\into \mu(x) |u|^q dx \right)^{\frac{1}{q}}.
\end{align*}
In the same way we define the space $L^q_\mu(\Omega;\R^N)$. By $W^{1,\mathcal{H}}(\Omega)$ we denote the corresponding Musielak-Orlicz Sobolev space which is defined by
\begin{align*}
    W^{1,\mathcal{H}}(\Omega)= \left \{u \in L^\mathcal{H}(\Omega) : |\nabla u| \in L^{\mathcal{H}}(\Omega) \right\}
\end{align*}
equipped with the norm
\begin{align*}
    \|u\|_{1,\mathcal{H}}= \|\nabla u \|_{\mathcal{H}}+\|u\|_{\mathcal{H}},
\end{align*}
where $\|\nabla u\|_\mathcal{H}=\|\,|\nabla u|\,\|_{\mathcal{H}}$. Since $W^{1,\mathcal{H}}(\Omega)$ is uniformly convex, see again Colasuonno-Squassina \cite[Proposition 2.14]{Colasuonno-Squassina-2016}, we know that it is a reflexive Banach space.

We have the following embedding results for the spaces $\Lp{\mathcal{H}}$ and $\Wp{\mathcal{H}}$.

\begin{proposition}\label{proposition_embeddings}
    Let \eqref{condition_poincare} be satisfied and let
    \begin{align}\label{critical_exponents}
	p^*:=\frac{Np}{N-p}
	\quad\text{and}\quad
	p_*:=\frac{(N-1)p}{N-p}
    \end{align}
    be the critical exponents to $p$, see \eqref{critical_exponent_domain} and \eqref{critical_exponent_boundary} for $r=p$. Then the following embeddings hold:
    \begin{enumerate}
	\item[(i)]
		$\Lp{\mathcal{H}} \hookrightarrow \Lp{r}$ and $\Wp{\mathcal{H}}\hookrightarrow \Wp{r}$ are continuous for all $r\in [1,p]$;
	\item[(ii)]
		$\Wp{\mathcal{H}} \hookrightarrow \Lp{r}$ is continuous for all $r \in [1,p^*]$;
	\item[(iii)]
		$\Wp{\mathcal{H}} \hookrightarrow \Lp{r}$ is compact for all $r \in [1,p^*)$;
	\item[(iv)]
		$\Wp{\mathcal{H}} \hookrightarrow \Lprand{r}$ is continuous for all $r \in [1,p_*]$;
	\item[(v)]
		$\Wp{\mathcal{H}} \hookrightarrow \Lprand{r}$ is compact for all $r \in [1,p_*)$;
	\item[(vi)]
		$\Lp{\mathcal{H}} \hookrightarrow L^q_\mu(\Omega)$ is continuous;
	\item[(vii)]
		$\Lp{q} \hookrightarrow \Lp{\mathcal{H}}$ is continuous.
    \end{enumerate}
\end{proposition}

For the continuity of the embedding $\Lp{\mathcal{H}} \hookrightarrow \Lp{r}$ we refer to Colasuonno-Squassina \cite[Propositions 2.3 and 2.15]{Colasuonno-Squassina-2016} while (ii)--(v) follow from the classical Sobolev embedding theorem and the trace embedding result. The statements (vi) and (vii) can be also found in Colasuonno-Squassina \cite[Propositions 2.15 (iv) and (v)]{Colasuonno-Squassina-2016}.

The norm $\|\cdot\|_{\mathcal{H}}$ and the modular function $\rho_\mathcal{H}$ are related as follows, see Liu-Dai \cite[Proposition 2.1]{Liu-Dai-2018}.

\begin{proposition}\label{proposition_modular_properties}
    Let \eqref{condition_poincare} be satisfied and let $\rho_{\mathcal{H}}$ be defined by \eqref{modular}.
    \begin{enumerate}
	\item[(i)]
	    If $y\neq 0$, then $\|y\|_{\mathcal{H}}=\lambda$ if and only if $ \rho_{\mathcal{H}}(\frac{y}{\lambda})=1$;
	\item[(ii)]
	    $\|y\|_{\mathcal{H}}<1$ (resp.\,$>1$, $=1$) if and only if $ \rho_{\mathcal{H}}(y)<1$ (resp.\,$>1$, $=1$);
	\item[(iii)]
	    If $\|y\|_{\mathcal{H}}<1$, then $\|y\|_{\mathcal{H}}^q\leqslant \rho_{\mathcal{H}}(y)\leqslant\|y\|_{\mathcal{H}}^p$;
	\item[(iv)]
	    If $\|y\|_{\mathcal{H}}>1$, then $\|y\|_{\mathcal{H}}^p\leqslant \rho_{\mathcal{H}}(y)\leqslant\|y\|_{\mathcal{H}}^q$;
	\item[(v)]
	    $\|y\|_{\mathcal{H}}\to 0$ if and only if $ \rho_{\mathcal{H}}(y)\to 0$;
	\item[(vi)]
	    $\|y\|_{\mathcal{H}}\to +\infty$ if and only if $ \rho_{\mathcal{H}}(y)\to +\infty$.
    \end{enumerate}
\end{proposition}

For $u \in \Wp{\mathcal{H}}$ let
\begin{align}\label{modular2}
    \hat{\rho}_\mathcal{H}(u) =\into \l(|\nabla u|^{p}+\mu(x)|\nabla u|^q\r)\,dx+\into \l(|u|^{p}+\mu(x)|u|^q\r)\,dx.
\end{align}

Following the proof of Liu-Dai \cite[Proposition 2.1]{Liu-Dai-2018} we have a similar result for the norm $\|\cdot\|_{1,\mathcal{H}}$ and the modular function $\hat{\rho}_\mathcal{H}$.

\begin{proposition}\label{proposition_modular_properties2}
    Let \eqref{condition_poincare} be satisfied and let $\hat{\rho}_{\mathcal{H}}$ be defined by \eqref{modular2}.
    \begin{enumerate}
	\item[(i)]
	    If $y\neq 0$, then $\|y\|_{1,\mathcal{H}}=\lambda$ if and only if $ \hat{\rho}_{\mathcal{H}}(\frac{y}{\lambda})=1$;
	\item[(ii)]
	    $\|y\|_{1,\mathcal{H}}<1$ (resp.\,$>1$, $=1$) if and only if $ \hat{\rho}_{\mathcal{H}}(y)<1$ (resp.\,$>1$, $=1$);
	\item[(iii)]
	    If $\|y\|_{1,\mathcal{H}}<1$, then $\|y\|_{1,\mathcal{H}}^q\leqslant \hat{\rho}_{\mathcal{H}}(y)\leqslant\|y\|_{1,\mathcal{H}}^p$;
	\item[(iv)]
	    If $\|y\|_{1,\mathcal{H}}>1$, then $\|y\|_{1,\mathcal{H}}^p\leqslant \hat{\rho}_{\mathcal{H}}(y)\leqslant\|y\|_{1,\mathcal{H}}^q$;
	\item[(v)]
	    $\|y\|_{1,\mathcal{H}}\to 0$ if and only if $ \hat{\rho}_{\mathcal{H}}(y)\to 0$;
	\item[(vi)]
	    $\|y\|_{1,\mathcal{H}}\to +\infty$ if and only if $ \hat{\rho}_{\mathcal{H}}(y)\to +\infty$.
    \end{enumerate}
\end{proposition}


We denote by $\lan\cdot,\cdot\ran_{\mathcal{H}}$ the duality pairing between $\Wp{\mathcal{H}}$ and its dual space $\Wp{\mathcal{H}}^*$ and consider the nonlinear operator $A\colon \Wp{\mathcal{H}}\to \Wp{\mathcal{H}}^*$ which is defined by
\begin{align}\label{operator_representation}
    \langle A(u),\ph\rangle_{\mathcal{H}} :=\into \left(|\nabla u|^{p-2}\nabla u+\mu(x)|\nabla u|^{q-2}\nabla u \right)\cdot\nabla\ph \,dx\quad\text{for all } u,\ph\in\Wp{\mathcal{H}}.
\end{align}
The properties of the operator $A\colon \Wp{\mathcal{H}}\to \Wp{\mathcal{H}}^*$ are stated in the following proposition, see Liu-Dai \cite{Liu-Dai-2018}.

\begin{proposition}
    The operator $A$ defined by \eqref{operator_representation} is bounded (that is, it maps bounded sets to bounded sets), continuous, strictly monotone (hence maximal monotone) and it is of type $(\Ss)_+$, that is,
    \begin{align*}
	u_n \weak u \text{ in }\Wp{\mathcal{H}}\quad\text{and}\quad \limsup_{n\to\infty} \langle A(u_n),u_n-u\rangle \leq 0,
    \end{align*}
    imply $u_n\to u$ in $\Wp{\mathcal{H}}$.
\end{proposition}

For $s \in \R$, we set $s^{\pm}=\max\{\pm s,0\}$ and for $u \in \Wp{\mathcal{H}}$ we define $u^{\pm}(\cdot)=u(\cdot)^{\pm}$. We have
\begin{align*}
    u^{\pm} \in \Wp{\mathcal{H}}, \quad |u|=u^++u^-, \quad u=u^+-u^-.
\end{align*}

If $X$ is a Banach space and $\ph\in C^1(X,\R)$, then we define
\begin{align*}
    K_\ph=\left\{u\in X \, : \, \ph'(u)=0\right\}
\end{align*}
being the critical set of $\ph$. Furthermore, we say that $\ph$ satisfies the Cerami condition, if every sequence $\{u_n\}_{n \geq 1} \subseteq X$ such that $\{\ph(u_n)\}_{n \geq 1}\subseteq \R$ is bounded and such that
\begin{align*}
    \left(1+\|u_n\|_X\right)\ph'(u_n) \to 0 \quad \text{in }X^* \text{ as }n \to \infty,
\end{align*}
admits a strongly convergent subsequence.

This compactness-type condition on the functional $\ph$ leads to a deformation theorem from which one can derive the minimax theory for the critical values of $\ph$. A central result of this theory is the so-called mountain pass theorem due to Ambrosetti-Rabinowitz \cite{Ambrosetti-Rabinowitz-1973} which we recall next.

\begin{theorem}\label{theorem_mountain_pass}
   Let $\varphi\in C^1(X)$ be a functional satisfying the $C$-condition and let $u_1,u_2 \in X, \|u_2-u_1\|_X> \rho>0$,
   \begin{align*}
	\max \{\varphi(u_1),\varphi(u_2)\}<\inf \{\varphi(u): \|u-u_1\|_X=\rho\}=:m_\rho
   \end{align*}
   and $c=\inf_{\gamma \in \Gamma} \max_{0 \leq t \leq 1} \varphi(\gamma(t))$ with $\Gamma=\{\gamma \in C\l([0,1],X\r): \gamma(0)=u_1, \gamma(1)=u_2\}$.
   Then $c \geq m_\rho$ with $c$ being a critical value of $\varphi$.
\end{theorem}

\section{A priori estimates for double phase problems} \label{section_3}

In this section we are going to prove the boundedness of weak solutions for double phase problems stated in a more general form than \eqref{problem}. For example, we allow in this section a convection term, that is, the dependence on the right-hand side on the gradient of the solution. We point out that such a result is of independent interest and can be applied for several model problems of this type. We consider the problem
\begin{equation}\label{problem2}
    \begin{aligned}
	-\divergenz\left(|\nabla u|^{p-2}\nabla u+\mu(x) |\nabla u|^{q-2}\nabla u\right) & =h_1(x,u,\nabla u)\quad && \text{in } \Omega,\\
	\left(|\nabla u|^{p-2}\nabla u+\mu(x) |\nabla u|^{q-2}\nabla u\right) \cdot \nu & = h_2(x,u) &&\text{on } \partial \Omega,
    \end{aligned}
\end{equation}
where we assume the following hypotheses on the data:
\begin{enumerate}[leftmargin=1.7cm]
    \item[H($h_1,h_2$)]			

	$h_1\colon\Omega\times\R\times\R^N\to \R$ and $h_2\colon\partial\Omega\times\R\to\R$ are Carath\'eodory functions satisfying
	\begin{align*}
	    |h_1(x,s,\xi)| & \leq a_1|\xi|^{p\frac{r_1-1}{r_1}}+a_2|s|^{r_1-1}+a_3 &&\text{for a.\,a.\,}x\in\Omega,\\
	    |h_2(x,s)| & \leq a_4 |s|^{r_2-1}+a_5&&\text{for a.\,a.\,}x\in\partial\Omega,
	\end{align*}
	for all $s \in \R$ and for all $\xi \in \R^N$ with positive constants $a_i$, $i\in \{1,\ldots , 5\}$, and $q<r_1\leq p^*$ as well as $q<r_2 \leq p_*$, where $p^*$ and $p_*$ are the critical exponents of $p$ stated in \eqref{critical_exponents}.
\end{enumerate}

We call $u\in \Wp{\mathcal{H}}$ a weak solution of problem \eqref{problem2} if
\begin{align}\label{weak_solution2}
    \into \left(|\nabla u|^{p-2}\nabla u+\mu(x)|\nabla u|^{q-2}\nabla u \right)\cdot\nabla v\,dx
    =\into h_1(x,u,\nabla u)v\,dx+\int_{\partial\Omega}h_2(x,u)v \,d\sigma
\end{align}
is satisfied for all test functions $v \in \Wp{\mathcal{H}}$.

Exploiting the recent result of Marino-Winkert \cite{Marino-Winkert-2019} we can prove the following result about the boundedness of weak solutions of \eqref{weak_solution2}.

\begin{theorem}\label{theorem_apriori}
    Let hypotheses \eqref{condition_poincare} and H($h_1,h_2$) be satisfied and let $u \in \Wp{\mathcal{H}}$ be a weak solution of problem \eqref{problem2}. Then, $u \in \Linf$.
\end{theorem}

\begin{proof}
    It is know that $u=u^+ -u^-$. Therefore, we can assume, without any loss of generality, that $u \geq 0$.

    Let $h>0$ and define $u_h:=\min\{u,h\}$. Choosing $v=uu_h^{\kappa p}$ with $\kappa>0$ as test function in \eqref{weak_solution2} we have
    \begin{align}\label{M1}
      \begin{split}
	&\into |\nabla u|^{p} u_h^{\kappa p}\,dx+ \kappa p\int_\Omega  |\nabla u|^{p-2}\nabla u\cdot \nabla u_h u_h^{\kappa p-1}u\,dx\\
	&\quad+\into \mu(x)|\nabla u|^{q}u_h^{\kappa p}\,dx+ \kappa p\int_\Omega \mu(x) |\nabla u|^{q-2}\nabla u\cdot \nabla u_h u_h^{\kappa p-1}u\,dx\\
	&= \into h_1(x,u,\nabla u) uu_h^{\kappa p} \,dx+\int_{\partial\Omega}h_2(x,u)uu_h^{\kappa p} \,d\sigma.
      \end{split}
    \end{align}
    Obviously, the third and the fourth integral on the left-hand side of \eqref{M1} are nonnegative. This gives
    \begin{align*}
      \begin{split}
	&\into |\nabla u|^{p} u_h^{\kappa p}\,dx+ \kappa p\int_\Omega  |\nabla u|^{p-2}\nabla u\cdot \nabla u_h u_h^{\kappa p-1}u\,dx\\
	&\leq  \into h_2(x,u,\nabla u) uu_h^{\kappa p} \,dx+uu_h^{\kappa p}+\int_{\partial\Omega}h_2(x,u)uu_h^{\kappa p} \,d\sigma.
      \end{split}
    \end{align*}
    Since $\Wp{\mathcal{H}}\subseteq \Wp{p}$ we can proceed exactly as in the proof of Theorem 3.1 of Marino-Winkert \cite[starting with (3.2)]{Marino-Winkert-2019} to obtain that $u \in \Linf$.
\end{proof}

\section{Constant sign solutions}\label{section_4}

In this section we are going to prove the existence of constant sign solutions of problem \eqref{problem}. First, we state our assumptions.

\begin{enumerate}
    \item[(H)]
	$f\colon \Omega \times \R \to \R$ and $g\colon\partial\Omega\times\R\to\R$ are Carath\'{e}odory functions such that the following hold:
	\begin{enumerate}
	    \item[(i)]
		There exist constants $c_1, c_2>0$ such that
		\begin{align*}
		|f(x,s)| & \leq c_1\left( 1 +|s|^{r_1-1}\right) \quad \text{for a.\,a.\,}x\in\Omega,\\
		|g(x,s)| & \leq c_2\left(1+|s|^{r_2-1}\right) \quad \text{for a.\,a.\,}x\in\partial\Omega,
		\end{align*}
		for all $s\in \R$, where $q<r_1<p^*$ and $q<r_2<p_*$ with the critical exponents $p^*$ and $p_*$ given in \eqref{critical_exponent_domain} and \eqref{critical_exponent_boundary}, respectively;
	    \item[(ii)]
		\begin{align*}
		    & \lim_{s \to \pm \infty}\,\frac{f(x,s)}{|s|^{q-2}s}=+\infty \quad\text{uniformly for a.\,a.\,}x\in\Omega;\\[1ex]
		    & \lim_{s \to \pm \infty}\,\frac{g(x,s)}{|s|^{q-2}s}=+\infty \quad\text{uniformly for a.\,a.\,}x\in\partial \Omega;
		\end{align*}
	    \item[(iii)]
		\begin{align*}
		    &\lim_{s \to 0}\,\frac{f(x,s)}{|s|^{p-2}s}=0 \quad\text{uniformly for a.\,a.\,}x\in\Omega;\\[1ex]
		    &\lim_{s \to 0}\,\frac{g(x,s)}{|s|^{p-2}s}=0 \quad\text{uniformly for a.\,a.\,}x\in\partial\Omega;
		\end{align*}
	    \item[(iv)]
		The functions
		\begin{align*}
		    s\mapsto f(x,s)s-qF(x,s)
		    \quad\text{and}\quad
		    s\mapsto g(x,s)s-qG(x,s)
		\end{align*}
		are nondecreasing on $\R_+$ and nonincreasing on $\R_-$ for a.\,a.\,$x\in\Omega$ and for a.\,a.\,$x\in\partial\Omega$, respectively, where
		\begin{align*}
		    F(x,s)=\int_0^s f(x,t)\,dt
		    \quad\text{and}\quad
		    G(x,s)=\int_0^sg(x,t)\,dt;
		\end{align*}

	    \item[(v)]
		The functions
		\begin{align*}
		    \frac{f(x,s)}{|s|^{q-1}}
		    \quad\text{and}\quad
		    \frac{g(x,s)}{|s|^{q-1}}
		\end{align*}
		are strictly increasing on $(-\infty,0)$ and on $(0,+\infty)$ for a.\,a.\,$x\in\Omega$ and for a.\,a.\,$x\in\partial\Omega$, respectively.
	\end{enumerate}
\end{enumerate}

Note that the continuity of $f(x,\cdot)$ and $g(x,\cdot)$ along with (H)(iii) implies that
\begin{align*}
    f(x,0)=0 \quad\text{for a.\,a.\,}x\in\Omega
    \quad\text{and}\quad
    g(x,0)=0 \quad\text{for a.\,a.\,}x\in\partial\Omega.
\end{align*}

We say that $u\in \Wp{\mathcal{H}}$ is a weak solution of problem \eqref{problem} if it satisfies
\begin{align*}
    &\into \left(|\nabla u|^{p-2}\nabla u+\mu(x)|\nabla u|^{q-2}\nabla u \right)\cdot\nabla v \,dx +\into \left(|u|^{p-2} u+\mu(x)|u|^{q-2}u \right)v \,dx\\
    &=\into f(x,u) v \,dx+\int_{\partial\Omega}g(x,u)v\,d\sigma
\end{align*}
for all test functions $v \in \Wp{\mathcal{H}}$.

The energy functional $\ph\colon \Wp{\mathcal{H}} \to \R$ corresponding to problem \eqref{problem} is defined by
\begin{align*}
    \begin{split}
	\ph(u)
	& =\frac{1}{p}\|\nabla u\|_p^p+\frac{1}{q}\|\nabla u\|_{q,\mu}^q+\frac{1}{p}\| u\|_p^p+\frac{1}{q}\|u\|_{q,\mu}^q-\into F(x,u)\,dx-\int_{\partial\Omega} G(x,u)\,d\sigma
    \end{split}
\end{align*}
for all $u \in \Wp{\mathcal{H}}$. Note that $\ph\in C^1(W^{1,\mathcal{H}}(\Omega))$, see Perera-Squassina \cite[Proposition 2.1]{Perera-Squassina-2018}, and that any $u\in K_{\ph}$ is a solution of problem \eqref{problem}.

First we want to produce two constant sign solutions. To this end, we consider the positive and negative truncations of the energy functional $\ph$. So, we consider $\ph_{\pm}\colon \Wp{\mathcal{H}} \to \R$ defined by
\begin{align*}
    \ph_{\pm}(u)
    & =\frac{1}{p}\|\nabla u\|_p^p+\frac{1}{q}\|\nabla u\|_{q,\mu}^q+\frac{1}{p}\| u\|_p^p+\frac{1}{q}\|u\|_{q,\mu}^q\\
    &\quad -\into F\l(x,\pm u^{\pm}\r)\,dx-\int_{\partial\Omega} G\l(x,\pm u^{\pm}\r)\,d\sigma.
\end{align*}

\begin{proposition}\label{proposition_cerami}
    Let hypotheses \eqref{condition_poincare} and (H) be satisfied. Then the functionals $\ph_{\pm}$ fulfill the Cerami condition.
\end{proposition}

\begin{proof}
    We will show the proof only for $\ph_+$, the proof for $\ph_-$ works in a similar way.

    Let $\{u_n\}_{n \geq 1} \subseteq \Wp{\mathcal{H}}$ be a sequence such that
    \begin{align}\label{c1}
	\left|\ph_+(u_n)\right| \leq M_1 \quad \text{for some }M_1>0 \text{ and for all }n\in\N
    \end{align}
    and
    \begin{align}\label{c2}
	\left(1+\|u_n\|_{1,\mathcal{H}}\right)\ph_+'(u_n)\to 0 \quad\text{in }\Wp{\mathcal{H}}^*.
    \end{align}
    Due to \eqref{c2} we have
    \begin{align}\label{c3}
	\begin{split}
	    &\left| \into |\nabla u_n|^{p-2}\nabla u_n \cdot \nabla v\,dx + \into \mu(x) |\nabla u_n|^{q-2}\nabla u_n \cdot \nabla v \,dx\right.\\
	    &\quad+\into |u_n|^{p-2} u_n v\,dx + \into \mu(x) |u_n|^{q-2}u_n v \,dx\\
	    & \left.\quad -\into f\l(x,u_n^+\r)v\,dx -\int_{\partial\Omega}g\l(x,u_n^+\r)v\,d\sigma\right| \leq \frac{\eps_n\|v\|_{1,\mathcal{H}}}{1+\|u_n\|_{1,\mathcal{H}}}
	\end{split}
    \end{align}
    for all $v \in \Wp{\mathcal{H}}$ with $\eps_n \to 0^+$. Taking $v=-u_n^- \in \Wp{\mathcal{H}}$ in \eqref{c3} we obtain
    \begin{align*}
	\l\|\nabla u_n^-\r\|_p^p+\l\|\nabla u_n^-\l\|_{q,\mu}^q +\r\|u_n^-\r\|_p^p+\l\|u_n^-\r\|_{q,\mu}^q&\leq \eps_n\quad\text{for all }n\in\N.
    \end{align*}
    Then, $\hat{\rho}_\mathcal{H}(u_n^-)\to 0$ as $n \to \infty$. Hence, by Proposition \ref{proposition_modular_properties2}(v) we have
    \begin{align*}
	\l\|u_n^-\r\|_{1,\mathcal{H}} \to 0\quad\text{as }n\to \infty.
    \end{align*}
    Thus
    \begin{align}\label{c4}
	u_n^- \to 0 \quad\text{in }\Wp{\mathcal{H}}.
    \end{align}
    Using \eqref{c1}  and \eqref{c4} we get
    \begin{align}\label{c5}
	\begin{split}
	    &\frac{q}{p}\l\|\nabla u_n^+\r\|_p^p+\l\|\nabla u_n^+\r\|_{q,\mu}^q+\frac{q}{p}\l\| u_n^+\r\|_p^p+\l\|u_n^+\r\|_{q,\mu}^q\\
	    & \quad -\into qF\l(x, u_n^+\r)\,dx-\int_{\partial\Omega} qG\l(x,u_n^+\r)\,d\sigma \leq M_2 \quad\text{for all }n\in\N
	\end{split}
    \end{align}
    for some $M_2>0$. We choose $v=u_n^+ \in \Wp{\mathcal{H}}$ in \eqref{c3} and obtain
    \begin{align}\label{c6}
	\begin{split}
	    & - \l\|\nabla u_n^+\r\|_p^p - \l\|\nabla u_n^+\r\|_{q,\mu}^q- \l\|u_n^+\r\|_p^p-\l\|u_n^+\r\|_{q,\mu}^q\\
	    & \quad +\into f\l(x,u_n^+\r)u_n^+\,dx +\int_{\partial\Omega}g\l(x,u_n^+\r)u_n^+\,d\sigma  \leq \eps_n \quad\text{for all }n\in\N.
	\end{split}
    \end{align}
    Now we add \eqref{c5} and \eqref{c6} to get
    \begin{align}\label{c7}
	\begin{split}
	    &\left(\frac{q}{p}-1\right)\l\|\nabla u_n^+\r\|_p^p +\left(\frac{q}{p}-1\right)\l\| u_n^+\r\|_p^p+\into\left(f\l(x,u_n^+\r)u_n^+- qF\l(x, u_n^+\r)\right)\,dx\\
	    & \qquad +\int_{\partial\Omega} \left(g\l(x,u_n^+\r)u_n^+-qG\l(x,u_n^+\r)\right)\,d\sigma \leq M_3 \quad\text{for all }n\in\N.
	\end{split}
    \end{align}

\medskip

\noindent
    {\bf Claim: } The sequence $\{u_n^+\}_{n \geq 1} \subseteq \Wp{\mathcal{H}}$ is bounded.

\medskip

    Arguing indirectly, we suppose, by passing to a subsequence if necessary, that
    \begin{align}\label{c7c}
	\l\|u_n^+\r\|_{1,\mathcal{H}} \to +\infty \quad\text{as }n\to +\infty.
    \end{align}
    Defining $y_n=\frac{u_n^+}{\l\|u_n^+\r\|_{1,\mathcal{H}}}$ for $n\in \N$ we see that $\|y_n\|_{1,\mathcal{H}}=1$ and $y_n \geq 0$ for all $n \in \N$. Thus, we may assume that
    \begin{align}\label{c8}
	y_n \weak y \quad \text{in }\Wp{\mathcal{H}}\quad\text{and}\quad y_n \to y \quad\text{in }\Lp{r_1} \text{ and } \Lprand{r_2},\quad \ y\geq 0,
    \end{align}
    see Proposition \ref{proposition_embeddings}(iii), (v).
	
\medskip

\noindent
    {\bf Case 1:} $y \neq 0$.

\medskip

    Let
    \begin{align*}
	\Omega_+=\left\{x\in \Omega : y(x)>0 \right\} \quad\text{and}\quad  \Gamma_+=\left\{x\in \partial\Omega : y(x)>0 \right\}.
    \end{align*}
    Of course, $|\Omega_+|_N>0$. Then, because of \eqref{c8} we have
    \begin{align*}
	&u_n^+(x) \to +\infty \quad \text{for a.\,a.\,}x\in\Omega_+
    \end{align*}
    and hence, due to (H)(ii),
    \begin{align}\label{c9}
	\frac{F(x,u_n^+(x))}{u_n^+(x)^q}\to +\infty \quad \text{for a.\,a.\,}x\in\Omega_+.
    \end{align}
    Applying \eqref{c9}, hypothesis (H)(ii) and Fatou's Lemma gives
    \begin{align}\label{c10}
	\int_{\Omega_+} \frac{F(x,u_n^+)}{\l\|u_n^+\r\|^q_{1,\mathcal{H}}}\,dx\to +\infty.
    \end{align}
    Furthermore, by (H)(i) and (ii) we have
    \begin{align}\label{c11}
	F(x,s) \geq -M_4  \quad \text{for a.\,a.\,}x\in\Omega, \text{ for all }s\in\R,
    \end{align}
    and for some $M_4>0$. From \eqref{c11} it follows
    \begin{align*}
	\int_{\Omega} \frac{F(x,u_n^+)}{\l\|u_n^+\r\|^q_{1,\mathcal{H}}}\,dx
	& = \int_{\Omega_+} \frac{F(x,u_n^+)}{\l\|u_n^+\r\|^q_{1,\mathcal{H}}}\,dx+\int_{\Omega\setminus \Omega_+} \frac{F(x,u_n^+)}{\l\|u_n^+\r\|^q_{1,\mathcal{H}}}\,dx\\
	& \geq \int_{\Omega_+} \frac{F(x,u_n^+)}{\l\|u_n^+\r\|^q_{1,\mathcal{H}}}\,dx - \frac{M_4}{\l\|u_n^+\r\|^q_{1,\mathcal{H}}}|\Omega|_N.
    \end{align*}
    Therefore, due to \eqref{c7c} and \eqref{c10}, we have
    \begin{align}\label{c12}
	\int_{\Omega} \frac{F(x,u_n^+)}{\l\|u_n^+\r\|^q_{1,\mathcal{H}}}\,dx\to +\infty.
    \end{align}
    If the Hausdorff surface measure of $\Gamma_+$ is positive, we can prove in a similar way that
    \begin{align}\label{c13}
	\int_{\partial\Omega} \frac{G(x,u_n^+)}{\l\|u_n^+\r\|^q_{1,\mathcal{H}}}\,d\sigma \to+ \infty,
    \end{align}
    or otherwise
    \begin{align}\label{c14}
	\int_{\partial\Omega} \frac{G(x,u_n^+)}{\l\|u_n^+\r\|^q_{1,\mathcal{H}}}\,d\sigma =0.
    \end{align}
    Thus, we obtain from \eqref{c12}, \eqref{c13} and \eqref{c14} that
    \begin{align}\label{c15}
	\int_{\Omega} \frac{F(x,u_n^+)}{\l\|u_n^+\r\|^q_{1,\mathcal{H}}}\,dx+\int_{\partial\Omega} \frac{G(x,u_n^+)}{\l\|u_n^+\r\|^q_{1,\mathcal{H}}}\,d\sigma\to +\infty.
    \end{align}

    On the other side we obtain from \eqref{c1} and \eqref{c4} that
    \begin{align*}
	&\int_{\Omega} \frac{F(x,u_n^+)}{\l\|u_n^+\r\|^q_{1,\mathcal{H}}}\,dx+\int_{\partial\Omega} \frac{G(x,u_n^+)}{\l\|u_n^+\r\|^q_{1,\mathcal{H}}}\,d\sigma\\
	&\leq \frac{1}{\l\|u_n^+\r\|_{1,\mathcal{H}}^{q-p}}\|\nabla y_n\|_p^p+\frac{1}{q} \|\nabla y_n\|_{q,\mu}^q+\frac{1}{\l\|u_n^+\r\|_{1,\mathcal{H}}^{q-p}}\|y_n\|_p^p+\frac{1}{q} \| y_n\|_{q,\mu}^q+M_5
    \end{align*}
    for all $n \in \N$ and for some $M_5>0$. This shows, because of $p<q$, \eqref{c7c} and $\|y_n\|_{1,\mathcal{H}}=1$ for all $n\in\N$, that
    \begin{align*}
	\int_{\Omega} \frac{F(x,u_n^+)}{\l\|u_n^+\r\|^q_{1,\mathcal{H}}}\,dx+\int_{\partial\Omega} \frac{G(x,u_n^+)}{\l\|u_n^+\r\|^q_{1,\mathcal{H}}}\,d\sigma \leq M_6 \quad\text{for all }n \in\N,
    \end{align*}
    for some $M_6>0$, which is a contradiction to \eqref{c15}.

\medskip

\noindent
    {\bf Case 2:} $y \equiv 0$.

\medskip

    Let $k \geq 1$ and put
    \begin{align*}
	v_n =(qk)^{\frac{1}{q}}y_n \quad\text{for all }n \in \N.
    \end{align*}
    By the definition of $y_n$ we have
    \begin{align}\label{c16}
	v_n \weak 0\quad \text{in }\Wp{\mathcal{H}}\quad\text{and}\quad v_n \to 0 \quad\text{in }\Lp{r_1} \text{ and } \Lprand{r_2}.
    \end{align}
    From \eqref{c16} it follows that
    \begin{align}\label{c17}
	\into F(x,v_n)\,dx \to 0 \quad\text{and}\quad \int_{\partial\Omega} G(x,v_n)\,d\sigma \to 0.
    \end{align}
    Recall that the energy functional $\ph\colon \Wp{\mathcal{H}}\to\R$ of problem \eqref{problem} is defined by
    \begin{align*}
	\ph(u)=\frac{1}{p}\|\nabla u\|_p^p+\frac{1}{q}\|\nabla u\|_{q,\mu}^q+\frac{1}{p}\| u\|_p^p+\frac{1}{q}\| u\|_{q,\mu}^q-\into F(x,u)\,dx-\int_{\partial\Omega}G(x,u)\,d\sigma.
    \end{align*}
    We have
    \begin{align}\label{c21}
	\ph(u)\leq  \ph_+(u) \quad\text{for all }u\in \Wp{\mathcal{H}} \text{ with }u \geq 0.
    \end{align}
    We choose $t_n \in [0,1]$ such that
    \begin{align}\label{c18}
	\ph\l(t_n u_n^+\r) =\max \left\{ \ph\l(tu_n^+\r): 0 \leq t \leq 1 \right\}.
    \end{align}
    Since $\l\|u_n^+\r\|_{1,\mathcal{H}}\to +\infty$ there exists $n_0 \in \N$ such that
    \begin{align}\label{c19}
	0 < \frac{(qk)^{\frac{1}{q}}}{\l\|u_n^+\r\|_{1,\mathcal{H}}} \leq 1 \quad\text{for all }n \geq n_0.
    \end{align}
    Applying \eqref{c18}, \eqref{c19}, Proposition \ref{proposition_modular_properties2}(ii) and \ref{c17} we obtain
    \begin{align*}
	\ph\l(t_nu_n^+\r)
	& \geq  \ph(v_n)\\
	& =\frac{1}{p}q^{\frac{p}{q}}k^{\frac{p}{q}} \|\nabla y_n\|_p^p + k \|\nabla y_n\|_{q,\mu}^q+\frac{1}{p}q^{\frac{p}{q}}k^{\frac{p}{q}} \| y_n\|_p^p + k \| y_n\|_{q,\mu}^q\\
	& \qquad -\into F(x,v_n)\,dx-\int_{\partial\Omega}G(x,v_n)\,d\sigma\\
	&  \geq \min\left\{\frac{1}{p}q^{\frac{p}{q}},1\right\}k^{\frac{p}{q}} \left[ \|\nabla y_n\|_p^p +  \|\nabla y_n\|_{q,\mu}^q+\| y_n\|_p^p + \| y_n\|_{q,\mu}^q\right]\\
	& \qquad -\into F(x,v_n)\,dx-\int_{\partial\Omega}G(x,v_n)\,d\sigma\\
	&= \min\left\{\frac{1}{p}q^{\frac{p}{q}},1\right\}k^{\frac{p}{q}} \hat{\rho}_\mathcal{H}(u) -\into F(x,v_n)\,dx-\int_{\partial\Omega}G(x,v_n)\,d\sigma\\
	& \geq \min\left\{\frac{1}{p}q^{\frac{p}{q}},1\right\}k^{\frac{p}{q}}-M_7 \quad \text{for all }n \geq n_1,
    \end{align*}
    for some $n_1 \geq n_0$. Since $k \geq 1$ is arbitrary, we conclude that
    \begin{align}\label{c20}
	\ph\l(t_nu_n^+\r) \to +\infty \quad\text{as }n \to\infty.
    \end{align}
    From \eqref{c1}, \eqref{c4} and \eqref{c21} we obtain
    \begin{align}\label{c22}
	\ph(0)=0 \quad\text{and}\quad \ph(u_n^+) \leq M_8\quad \text{for all }n \in \N,
    \end{align}
    for some $M_8>0$. Combining \eqref{c20} and \eqref{c22} gives
    \begin{align}\label{c23}
	t_n \in (0,1) \quad\text{for all }n \geq n_2,
    \end{align}
    for some $n_2 \geq n_1$. By the chain rule, \eqref{c23} and \eqref{c18} imply that
    \begin{align*}
	0=\frac{d}{dt}\ph\l(tu_n^+\r)\Big|_{t=t_n} = \l\lan \ph' \l(t_n u_n^+\r),u_n^+\r\ran \quad\text{for all }n  \geq n_2.
    \end{align*}
    This means
    \begin{align}\label{c24}
	\begin{split}
	    & \l\|\nabla\l( t_n u_n^+\r)\r\|_p^p+\l\|\nabla \l(t_nu_n^+\r)\r\|_{q,\mu}^q+  \l\| t_n u_n^+\r\|_p^p+\l\|t_nu_n^+\r\|_{q,\mu}^q\\
	    &=\into f\left(x,t_nu_n^+\right)t_nu_n^+\,dx+\int_{\partial\Omega} g\left(x,t_nu_n^+\right)t_nu_n^+\,d\sigma
	\end{split}
    \end{align}
    for all $n \geq n_2$. By hypothesis (H)(iv) and \eqref{c7} we obtain
    \begin{align*}
	&\left(\frac{q}{p}-1\right)\l\|\nabla \l(t_nu_n^+\r)\r\|_p^p +\left(\frac{q}{p}-1\right)\l\| t_nu_n^+\r\|_p^p\\
	& \quad+\into\left( f\l(x,t_nu_n^+\r)t_nu_n^+-qF\l(x,t_nu_n^+\r)\right)\,dx+\int_{\partial\Omega}\left( g\l(x,t_nu_n^+\r)t_nu_n^+-qG\l(x,t_nu_n^+\r)\right)\,d\sigma\\
	&\leq \left(\frac{q}{p}-1\right)\l\|\nabla \l(t_nu_n^+\r)\r\|_p^p +\left(\frac{q}{p}-1\right)\l\| t_nu_n^+\r\|_p^p\\
	& \quad+\into\left( f\l(x,u_n^+\r)u_n^+-qF\l(x,u_n^+\r) \right)\,dx +\int_{\partial\Omega}\left(g\l(x,u_n^+\r)u_n^+-qG\l(x,u_n^+\r) \right)\,d\sigma\\
	&\leq \left(\frac{q}{p}-1\right)\l\|\nabla u_n^+\r\|_p^p +\left(\frac{q}{p}-1\right)\l\| u_n^+\r\|_p^p\\
	& \qquad +\into\left( f\l(x,u_n^+\r)u_n^+-qF\l(x,u_n^+\r) \right)\,dx +\int_{\partial\Omega}\left(g\l(x,u_n^+\r)u_n^+-qG\l(x,u_n^+\r) \right)\,d\sigma\\
	& \leq M_3
    \end{align*}
    for all $ n \geq n_3$. This gives
    \begin{align}\label{c25}
	\begin{split}
	    &\left(\frac{q}{p}-1\right)\l\|\nabla \l(t_nu_n^+\r)\r\|_p^p +\left(\frac{q}{p}-1\right)\l\| t_nu_n^+\r\|_p^p\\
	    &+\into f\l(x,t_nu_n^+\r)t_nu_n^+\,dx+\int_{\partial\Omega} g\l(x,t_nu_n^+\r)t_nu_n^+\,d\sigma\\
	    & \leq \into qF\l(x,t_nu_n^+\r)\,dx+\int_{\partial\Omega}qG\l(x,t_nu_n^+\r)\,d\sigma+M_3.
	\end{split}
    \end{align}
    Combining \eqref{c24} and \eqref{c25} leads to
    \begin{align*}
	\begin{split}
	    & \frac{q}{p}\l\|\nabla\l( t_n u_n^+\r)\r\|_p^p+\l\|\nabla \l(t_nu_n^+\r)\r\|_{q,\mu}^q+ \frac{q}{p} \l\| t_n u_n^+\r\|_p^p+\l\|t_nu_n^+\r\|_{q,\mu}^q\\
	    &\qquad -\into qF\l(x,t_nu_n^+\r)\,dx-\int_{\partial\Omega}qG\l(x,t_nu_n^+\r)\,d\sigma\\
	    & \leq M_{3},
	\end{split}
    \end{align*}
    for all $n \geq n_3$, which implies
    \begin{align*}
	q \ph\l(t_nu_n^+\r) \leq M_3 \quad\text{for all } n \geq n_3.
    \end{align*}
    This contradicts \eqref{c20} and so the claim is proved.

\medskip

    From \eqref{c4} and the Claim we know that the sequence $\{u_n\}_{n\geq 1}\subseteq \Wp{\mathcal{H}}$ is bounded. Therefore we may assume that
    \begin{align}\label{c27}
	u_n \weak u \quad\text{in }\Wp{\mathcal{H}}\quad\text{and}\quad u_n \to u \quad\text{in }\Lp{r_1} \text{ and } \Lprand{r_2}.
    \end{align}
    Due to \eqref{c27} we have
    \begin{align}\label{c28}
	\nabla u_n \weak \nabla u \quad\text{in } L^q_\mu\l(\Omega;\R^N\r) \quad\text{and}\quad \nabla u_n \weak \nabla u \quad\text{in }L^p\l(\Omega;\R^N\r).
    \end{align}
    Taking $v=u_n-u \in \Wp{\mathcal{H}}$ in \eqref{c3}, passing to the limit as $n\to\infty$ and using \eqref{c27} we obtain
    \begin{align}\label{c29}
	\|\nabla u_n\|_{ q,\mu} \to \|\nabla u\|_{q,\mu}\quad\text{and}\quad \|\nabla u_n\|_p \to \|\nabla u\|_p.
    \end{align}
    Since the spaces $ L^q_\mu\l(\Omega;\R^N\r)$ and $L^p\l(\Omega;\R^N\r)$ are uniformly convex, we know that they satisfy the Kadec-Klee property, see Gasi\'nski-Papageorgiou \cite[p.\,911]{Gasinski-Papageorgiou-2006}. Hence, from \eqref{c28} and \eqref{c29} it follows that
    \begin{align*}
	\nabla u_n \to \nabla u \quad\text{in } L^q_\mu\l(\Omega;\R^N\r) \quad\text{and}\quad \nabla u_n \to \nabla u \quad\text{in }L^p\l(\Omega;\R^N\r).
    \end{align*}
    Hence, by Proposition \ref{proposition_modular_properties}(ii) we conclude that
    \begin{align*}
	\|u_n - u\|_{1,\mathcal{H}} \to 0.
    \end{align*}
    Thus, $\ph_+$ fulfills the Cerami condition.
\end{proof}

The following proposition will be useful for later considerations.

\begin{proposition}\label{proposition_auxiliary_result}
    Let hypotheses \eqref{condition_poincare} and (H) be satisfied. Then for each $\eps>0$ there exist $\hat{c},\tilde{c}_\eps, \hat{c}_\eps >0$ such that
    \begin{align*}
	\ph(u),\, \ph_{\pm}(u) \geq
	\begin{cases}
	    \hat{c} \|u\|_{1,\mathcal{H}}^q - \tilde{c}_\eps \|u\|^{r_1}_{1,\mathcal{H}}- \hat{c}_\eps \|u\|_{1,\mathcal{H}}^{r_2} & \text{if } \|u\|_{1,\mathcal{H}} \leq 1,\\
	    \hat{c} \|u\|_{1,\mathcal{H}}^p - \tilde{c}_\eps \|u\|^{r_1}_{1,\mathcal{H}}- \hat{c}_\eps \|u\|_{1,\mathcal{H}}^{r_2} & \text{if } \|u\|_{1,\mathcal{H}}> 1.
	\end{cases}
    \end{align*}
\end{proposition}

\begin{proof}
    We will show the proof only for the functional $\ph$, the proofs for the other functionals work in a similar way.

    Taking hypotheses (H)(i), (iii) into account, for a given $\eps>0$, there exist $\hat{c}_1=\hat{c}_1(\eps)>0$ and $\hat{c}_2=\hat{c}_2(\eps)>0$ such that
    \begin{align}\label{l1}
	\begin{split}
	    F(x,s) &\leq \frac{\eps}{p}|s|^p +\hat{c}_1|s|^{r_1} \quad \text{for a.\,a.\,}x\in \Omega,\\
	    G(x,s) &\leq \frac{\eps}{p}|s|^p +\hat{c}_2|s|^{r_2} \quad \text{for a.\,a.\,}x\in \partial \Omega.
	\end{split}
    \end{align}
    Let $u \in \Wp{\mathcal{H}}$. Applying \eqref{l1}, the Sobolev and trace embeddings for $\Wp{p}$ along with Propositions \ref{proposition_embeddings}(ii), (iii) and \ref{proposition_modular_properties}(c) we obtain
    \begin{align*}
	\begin{split}
	    &\ph(u)\\
	    & \geq \frac{1}{p}\|\nabla u\|_p^p+\frac{1}{q}\|\nabla u\|_{q,\mu}^q+\frac{1}{p}\| u\|_p^p+\frac{1}{q}\|u\|_{q,\mu}^q\\
	    & \qquad -\frac{\eps}{p} \|u\|_p^p - \hat{c}_1 \|u\|_{r_1}^{r_1}-\frac{\eps}{p} \|u\|_{p,\partial\Omega}^p - \hat{c}_2\|u\|_{r_2,\partial\Omega}^{r_2}\\
	    & \geq \frac{1}{p} \left[1- \l(C_\Omega^p +C_{\partial\Omega}^p\r)\eps\right] \|\nabla u\|_p^p+\frac{1}{q}\|\nabla u\|_{q,\mu}^q\\
	    & \qquad +\frac{1}{p} \left[1- (C_\Omega^p +C_{\partial\Omega}^p)\eps\right]\| u\|_p^p+\frac{1}{q}\|u\|_{q,\mu}^q- \hat{c}_1 \l(C_\Omega^{\mathcal{H}}\r)^{r_1} \|u\|_{1,\mathcal{H}}^{r_1}- \hat{c}_2\l(C_{\partial\Omega}^{\mathcal{H}}\r)^{r_2}\|u\|_{1,\mathcal{H}}^{r_2}\\
	    & \geq \min \l\{\frac{1}{p} \left[1- \l(C_\Omega^p +C_{\partial\Omega}^p\r)\eps\right],\frac{1}{q}\r\} \hat{\rho}_{\mathcal{H}}(u)- \hat{c}_1 \l(C_\Omega^{\mathcal{H}}\r)^{r_1} \|u\|_{1,\mathcal{H}}^{r_1}- \hat{c}_2\l(C_{\partial\Omega}^{\mathcal{H}}\r)^{r_2}\|u\|_{1,\mathcal{H}}^{r_2},
	\end{split}
    \end{align*}
    where $C_\Omega$ and $C_{\partial\Omega}$ are the embedding constants from the embeddings $\Wp{p}\to \Lp{p}$ and $\Wp{p} \to \Lprand{p}$ respectively,
    while $C_\Omega^{\mathcal{H}}$ and $C_{\partial\Omega}^{\mathcal{H}}$ are the embedding constants from the embeddings 
    $\Wp{\mathcal{H}} \to \Lp{r_1}$ and $\Wp{\mathcal{H}} \to \Lprand{r_2}$, respectively.

    Choosing $\eps$ such that $\eps \in \left(0,\frac{1}{C_\Omega^p +C_{\partial\Omega}^p} \right)$ and applying Proposition \ref{proposition_modular_properties2}(iii), (iv) we get the assertion of the proposition with
    \begin{align*}
	\hat{c}=\min \l\{\frac{1}{p} \left[1- \l(C_\Omega^p +C_{\partial\Omega}^p\r)\eps\right],\frac{1}{q}\r\}, \quad \tilde{c}_\eps=\hat{c}_1 \l(C_\Omega^{\mathcal{H}}\r)^{r_1}, \quad  \hat{c}_\eps=\hat{c}_2\l(C_{\partial\Omega}^{\mathcal{H}}\r)^{r_2}.
    \end{align*}
\end{proof}

Now it is easy to show that $u=0$ is a local minimizer of the functionals $\ph_{\pm}$.

\begin{proposition}\label{proposition_local_minimizers}
    Let hypotheses \eqref{condition_poincare} and (H) be satisfied. Then $u=0$ is a local minimizer for both functionals $\ph_{\pm}$.
\end{proposition}

\begin{proof}
    As before, we will show the proof only for the functional $\ph_+$, the proof for $\ph_-$ is working in a similar way. Let $u \in \Wp{\mathcal{H}}$ be such that $\|u\|_{1,\mathcal{H}}<1$. Applying Proposition \ref{proposition_auxiliary_result} gives
    \begin{align*}
	\begin{split}
	    \ph_+(u)
	    \geq \hat{c} \|u\|_{1,\mathcal{H}}^q - \tilde{c}_\eps \|u\|^{r_1}_{1,\mathcal{H}}- \hat{c}_\eps \|u\|_{1,\mathcal{H}}^{r_2}.
	\end{split}
    \end{align*}
    Since $q<r_1, r_2$ there exists $\eta \in (0,1)$ small enough such that
    \begin{align*}
	\ph_+(u) >0=\ph_+(0) \quad\text{for all } u \in \Wp{\mathcal{H}} \text{ with }0<\|u\|_{1,\mathcal{H}}<\eta.
    \end{align*}
    Hence,  $u=0$ is a (strict) local minimizer of $\ph_+$.
\end{proof}

The following proposition is a direct consequence of hypothesis (H)(ii).

\begin{proposition}\label{proposition_unbounded_below}
    Let hypotheses \eqref{condition_poincare} and (H) be satisfied. Then, for $u \in \Wp{\mathcal{H}}$ with $u(x)>0$ for a.\,a.\,$x\in\Omega$, it holds $\ph_{\pm}(tu)\to -\infty$ as $t \to \pm\infty$.
\end{proposition}

Now we are ready to prove the existence of bounded constant sign solutions for problem \eqref{problem}.

\begin{proposition}\label{proposition_constant_sign_solutions}
    Let hypotheses \eqref{condition_poincare} and (H) be satisfied. Then problem \eqref{problem} has at least two nontrivial constant sign solutions $u_0, v_0\in \Wp{\mathcal{H}}\cap \Linf$ such that
	\begin{align*}
		u_0(x) \geq 0 \quad\text{and}\quad v_0(x) \leq 0 \quad \text{for a.\,a.\,} x\in \Omega.
	\end{align*}
\end{proposition}

\begin{proof}
    From Propositions \ref{proposition_local_minimizers} and Papageorgiou-R\u{a}dulescu-Repov\v{s} \cite[Theorem 5.7.6]{Papageorgiou-Radulescu-Repovs-2019} there exist $\eta_{\pm} \in (0,1)$ small enough such that
    \begin{align}\label{const1}
	\ph_{\pm}(0)=0<\inf\left\{\ph_{\pm}(0): \|u\|_{1,\mathcal{H}}=\eta_{\pm}\right\}=m_{\pm}.
    \end{align}
    By \eqref{const1} and the Propositions \ref{proposition_cerami} and \ref{proposition_unbounded_below} we are able to use the mountain pass theorem (see Theorem \ref{theorem_mountain_pass}) which implies the existence of $u_0, v_0 \in \Wp{\mathcal{H}}$ such that $u_0 \in K_{\ph_+}, \ v_0 \in K_{\ph_-}$ and
    \begin{align*}
	\ph_+(0)=0<m_+ \leq \ph_+(u_0)
	\quad \text{as well as}\quad
	\ph_-(0)=0<m_-\leq \ph_-(v_0).
    \end{align*}
    This shows that $u_0 \neq 0$ and $v_0\neq 0$. Moreover, we have $\ph_+'(u_0)=0$ which means that
    \begin{align*}
	&\into \left(|\nabla u_0|^{p-2}\nabla u_0+\mu(x)|\nabla u_0|^{q-2}\nabla u_0 \right)\cdot\nabla v \,dx +\into \left(|u_0|^{p-2} u_0+\mu(x)|u_0|^{q-2}u_0 \right)v \,dx\\
	&=\into f(x,u_0^+) v \,dx+\int_{\partial\Omega}g(x,u_0^+)v\,d\sigma
    \end{align*}
    for all $v \in \Wp{\mathcal{H}}$. Choosing $v=-u_0^-\in \Wp{\mathcal{H}}$ we obtain
    \begin{align*}
	\hat{\rho}_{\mathcal{H}}(u_0^-)=0
    \end{align*}
    and so, by Proposition \ref{proposition_modular_properties2}, we have
    \begin{align*}
	\|u_0^-\|_{1,\mathcal{H}}=0.
    \end{align*}
    Therefore, $u_0 \geq 0, u_0\neq 0$. In the same way we can show that $v_0 \leq 0, v_0 \neq 0$. Finally, by applying Theorem \ref{theorem_apriori}, we have that $u_0, v_0 \in \Linf$.
\end{proof}

\section{Sign changing solution}\label{section_5}

In this section we are interested in the existence of a sign-changing solution of problem \eqref{problem}. Following the treatment of Liu-Wang-Wang \cite{Liu-Wang-Wang-2004} and Gasi\'nski-Papageorgiou \cite{Gasinski-Papageorgiou-2019} we introduce the so-called Nehari manifold for the functional $\ph$ which is defined by
\begin{align*}
    N=\Big\{u \in \Wp{\mathcal{H}}: \lan \ph'(u),u\ran=0, \ u\neq 0\Big\}.
\end{align*}
Since we are interested in sign-changing solutions, we also need the following set
\begin{align*}
    N_0=\Big\{u \in \Wp{\mathcal{H}}: u^+ \in N, \ -u^-\in N \Big\}.
\end{align*}

\begin{proposition}\label{proposition_nodal_1}
    Let hypotheses \eqref{condition_poincare} and (H) be satisfied. Let $u \in \Wp{\mathcal{H}}$, $u \neq 0$, then there exists a unique $t_0=t_0(u)>0$ such that $t_0u \in N$.
\end{proposition}

\begin{proof}
    Let $\zeta_u\colon (0,+\infty)\to \R$ be defined by
    \begin{align}\label{n5}
	\begin{split}
	    \zeta_u(t)
	    &=\l\lan \ph'(tu),u\r\ran \\
	    & =t^{p-1}\|\nabla u\|_p^p+t^{q-1}\|\nabla u\|_{q,\mu}^q+t^{p-1}\|u\|_p^p+t^{q-1}\| u\|_{q,\mu}^q\\
	    & \quad -\into f(x,tu)u\,dx-\int_{\partial \Omega} g(x,tu)u\,d\sigma.
	\end{split}
    \end{align}
    By hypothesis (H)(v) we have for $t\in (0,1)$ and $|u(x)|>0$
    \begin{align*}
	& \frac{f(x,tu)(tu)}{t^q|u|^q} \leq \frac{f(x,u)u}{|u|^q} \quad \text{for a.\,a.\,}x\in \Omega,\\
	& \frac{g(x,tu)(tu)}{t^q|u|^q} \leq \frac{g(x,u)u}{|u|^q} \quad \text{for a.\,a.\,}x\in \partial\Omega,
    \end{align*}
    which implies
    \begin{align}\label{n4}
	\begin{split}
	    &f(x,tu)u \leq t^{q-1}f(x,u)u \quad \text{for a.\,a.\,}x\in \Omega,\\
	    &g(x,tu)u \leq t^{q-1}g(x,u)u \quad \text{for a.\,a.\,}x\in \partial\Omega.
	\end{split}
    \end{align}
    From \eqref{n5} and \eqref{n4} we obtain
    \begin{align*}
	\begin{split}
	    \zeta_u(t)
	    &\geq t^{p-1}\|\nabla u\|_p^p+t^{p-1}\|u\|_p^p\\
	    & \quad -t^{q-1}\into f(x,u)u\,dx-t^{q-1}\int_{\partial \Omega} g(x,u)u\,d\sigma.
	\end{split}
    \end{align*}
    Therefore, since $p<q$,
    \begin{align}\label{n6}
	\zeta_u(t) >0 \quad\text{for small }t\in (0,1).
    \end{align}

    On the other hand, we have for $t>0$
    \begin{align}\label{n7}
	\begin{split}
	    \frac{\zeta_u(t)}{t^{q-1}}
	    &=\frac{1}{t^{q-p}}\|\nabla u\|_p^p+\|\nabla u\|_{q,\mu}^q+\frac{1}{t^{q-p}}\| u\|_p^p+\|u\|_{q,\mu}^q\\
	    &\quad -\into \frac{f(x,tu)}{t^{q-1}}u\,dx-\int_{\partial\Omega} \frac{g(x,tu)}{t^{q-1}}u\,d\sigma.
	\end{split}
    \end{align}
    Applying hypothesis (H)(ii) and passing to the limit in \eqref{n7} as $t\to +\infty$ gives
    \begin{align*}
	\lim_{t\to +\infty }\frac{\zeta_u(t)}{t^{q-1}}=-\infty,
    \end{align*}
    as $p<q$. Hence
    \begin{align}\label{n8}
	\zeta_u(t) <0 \quad\text{for }t>0 \text{ large enough}.
    \end{align}
    Then, from \eqref{n6}, \eqref{n8} and the intermediate value theorem there exists $t_0=t_0(u)>0$ such that
    \begin{align*}
	\zeta_u(t_0)=0,
    \end{align*}
    which implies
    \begin{align*}
	\l\lan \ph'(t_0u),t_0u\r\ran =0.
    \end{align*}
    Hence
    \begin{align*}
	t_0u \in N.
    \end{align*}
    Note that equation $\zeta_u(t)=0$ can be equivalently written as
    \begin{align*}
	-\|\nabla u\|_{q,\mu}^q-\|u\|_{q,\mu}^q
	&=\frac{1}{t^{q-p}}\|\nabla u\|_p^p+\frac{1}{t^{q-p}}\| u\|_p^p\\
	&\quad -\into \frac{f(x,tu)(tu)}{t^{q}}\,dx-\int_{\partial\Omega} \frac{g(x,tu)(tu)}{t^{q}}\,d\sigma.
    \end{align*}
    The right-hand side of this inequality is strictly increasing in $t>0$. Therefore, there exists a unique $t_0=t_0(u)$ such that
    \begin{align*}
	\zeta_u(t_0)=0.
    \end{align*}
\end{proof}

\begin{proposition}\label{proposition_nodal_2}
    Let hypotheses \eqref{condition_poincare} and (H) be satisfied. Let $u \in N$, then $\ph(tu) \leq \ph(u)$ for all $t>0$ (with strict inequality when $t\ne 1$).
\end{proposition}

\begin{proof}
    Let $k_u\colon (0,\infty)\to \R$ be defined by
    \begin{align*}
	k_u(t)=\ph(tu) \quad\text{for all }t>0.
    \end{align*}
    Because $u\in N$, it holds
    \begin{align}\label{n9}
	k_u'(1)=0,
    \end{align}
    which is, due to Proposition \ref{proposition_nodal_1}, the unique critical point of $k_u$. From hypotheses (H)(i), (ii), there exists, for any given $\tau>0$, a constant $c_\tau>0$ such that
    \begin{align}\label{n10}
	\begin{split}
	    & F(x,s) \geq \frac{\tau}{q} |s|^q-c_\tau\quad\text{for a.\,a.\,}x\in\Omega \text{ and all }s\in\R,\\
	    & G(x,s) \geq \frac{\tau}{q} |s|^q -c_\tau\quad\text{for a.\,a.\,}x\in\partial \Omega \text{ and all }s\in\R.
	\end{split}
    \end{align}
    Taking \eqref{n10} into account, we have for $t>0$
    \begin{align*}
	k_u(t)
	& = \ph(tu)\\
	& \leq \frac{t^p}{p} \|\nabla u\|_p^p +\frac{t^q}{q} \|\nabla u\|_{q,\mu}^q+\frac{t^p}{p} \|u\|_p^p +\frac{t^q}{q} \| u\|_{q,\mu}^q\\
	& \quad -\frac{\tau t^q}{q}\|u\|_q^q-\frac{\tau t^q}{q}\|u\|_{q,\partial\Omega}^q+c_\tau\l(|\Omega|_N+|\partial\Omega|_N\r)\\
	& = \frac{t^p}{p} \left(\|\nabla u\|_p^p+\|u\|_p^p\right)+\frac{t^q}{q}\left(\|\nabla u\|_{q,\mu}^q+\| u\|_{q,\mu}^q-\tau \left(\|u\|_q^q+ \|u\|_{q,\partial\Omega}^q\right)\right)\\
	&\quad +c_\tau\l(|\Omega|_N+|\partial\Omega|_N\r).
    \end{align*}
    Taking $\tau$ large enough we have
    \begin{align*}
	\ph(tu) \leq c_3 t^p-c_4t^q
    \end{align*}
    for some $c_3, c_4>0$. Since $p<q$ we obtain
    \begin{align}\label{n11}
	k_u(t) =\ph(tu)<0 \quad\text{for $t>0$ large enough}.
    \end{align}

    Applying Proposition \ref{proposition_auxiliary_result}, for $t>0$ small enough we obtain
    \begin{align*}
	k_u(t)& =\ph(tu)\\
	& \geq \hat{c} \|u\|_{1,\mathcal{H}}^q - \tilde{c}_\eps \|u\|^{r_1}_{1,\mathcal{H}}- \hat{c}_\eps \|u\|_{1,\mathcal{H}}^{r_2}\\
	& = c_5t^q-c_6t^{r_1}-c_7t^{r_2}
    \end{align*}
    for some $c_5, c_6, c_7>0$. Since $q<r_1,r_2$ we conclude that
    \begin{align}\label{n12}
	k_u(t) =\ph(tu)>0 \quad\text{for $t>0$ small enough}.
    \end{align}
    From \eqref{n11} and \eqref{n12} we know that there exists a local minimizer $t_0(u)>0$ of $k_u$. Since $t=1$ is the only critical point of $k_u$, see \eqref{n9}, we have that $t_0(u)=1$ which is a global minimizer of $k_u$. Hence, we have
    \begin{align*}
	k_u(t) \leq k_u(1) \quad\text{for all }t > 0
    \end{align*}
    and so
    \begin{align*}
	\ph(tu) \leq \ph(u) \quad\text{for all }t > 0.
    \end{align*}
\end{proof}

\begin{proposition}\label{proposition_coercive}
    Let hypotheses \eqref{condition_poincare} and (H) be satisfied. Then the functional $\ph\big|_N$ is coercive.
\end{proposition}

\begin{proof}
    It is enough to show that if $\{u_n\}_{n \geq 1} \subseteq N$ and
    \begin{align}\label{n1}
	\ph(u_n) \leq M_9 \quad \text{for all }n\in\N
    \end{align}
    for some $M_9>0$, then the sequence $\{u_n\}_{n \geq 1} \subseteq \Wp{\mathcal{H}}$ is bounded.

    Supposing the opposite we can assume that $\|u_n\|_{1,\mathcal{H}}\to+\infty$. Letting $y_n=\frac{u_n}{\|u_n\|_{1,\mathcal{H}}}$ we can assume that $y_n\weak y$ in $\Wp{\mathcal{H}}$. Suppose that $y=0$. Since $u_n \in N$ and $y_n \weak 0$ we have for each $t>0$ that
    \begin{align*}
	\begin{split}
	    \ph(u_n)
	    & \geq \ph(t y_n)\\
	    & =\frac{1}{p}\|\nabla (t y_n)\|_p^p+\frac{1}{q}\|\nabla (t y_n)\|_{q,\mu}^q+\frac{1}{p}\| t y_n\|_p^p+\frac{1}{q}\|t y_n\|_{q,\mu}^q\\
	    & \quad -\into F(x,t y_n)\,dx-\int_{\partial\Omega} G(x,t y_n)\,d\sigma\\
	    & \geq \frac{1}{q} \|t y_n\|_{1,\mathcal{H}}^p-\into F(x,t y_n)\,dx-\int_{\partial\Omega} G(x,t y_n)\,d\sigma\to \frac{1}{q}t^p,
	\end{split}
    \end{align*}
    since $\| y_n\|_{1,\mathcal{H}}^p=1$ where we have used Propositions \ref{proposition_modular_properties2} and \ref{proposition_nodal_2}. Taking $t>0$ large enough we get a contradiction with \eqref{n1}. Hence, $y\neq 0$. Applying Proposition \ref{proposition_modular_properties2} we have
    \begin{align}\label{n14}
	\begin{split}
	    \ph(u_n)
	    & \leq \frac{1}{p}\|\nabla u_n\|_p^p+\frac{1}{q}\|\nabla u_n\|_{q,\mu}^q+\frac{1}{p}\| u_n\|_p^p+\frac{1}{q}\|u_n\|_{q,\mu}^q\\
	    & \quad -\into F(x,\|u_n\|_{1,\mathcal{H}} y_n)\,dx-\int_{\partial\Omega} G(x,\|u_n\|_{1,\mathcal{H}} y_n)\,d\sigma\\
	    & \leq \frac{1}{p} \|u_n\|_{1,\mathcal{H}}^q -\into F(x,\|u_n\|_{1,\mathcal{H}} y_n)\,dx-\int_{\partial\Omega} G(x,\|u_n\|_{1,\mathcal{H}} y_n)\,d\sigma.
	\end{split}
    \end{align}
    Dividing \eqref{n14} by $\|u_n\|_{1,\mathcal{H}}^q$, passing to the limit as $n\to \infty$ and applying (H)(ii), we obtain $\frac{\ph(u_n)}{\|u_n\|_{1,\mathcal{H}}^q}\to -\infty$ which contradicts $\ph(u_n) \geq 0$, see Proposition \ref{proposition_nodal_2} . This proves the coercivity of $\ph\big|_N$.
\end{proof}

Let $m=\inf\limits_N \ph$ and $m_0=\inf\limits_{N_0}  \ph$. First, we show that $m>0$.

\begin{proposition}\label{proposition_positive_infimum_1}
    Let hypotheses \eqref{condition_poincare} and (H) be satisfied. Then $m>0$.
\end{proposition}

\begin{proof}
    Recall the statement of Proposition \ref{proposition_auxiliary_result}, namely,
    \begin{align*}
	\ph(u) \geq
	\begin{cases}
	    \hat{c} \|u\|_{1,\mathcal{H}}^q - \tilde{c}_\eps \|u\|^{r_1}_{1,\mathcal{H}}- \hat{c}_\eps \|u\|_{1,\mathcal{H}}^{r_2} & \text{if } \|u\|_{1,\mathcal{H}} \leq 1,\\
	    \hat{c} \|u\|_{1,\mathcal{H}}^p - \tilde{c}_\eps \|u\|^{r_1}_{1,\mathcal{H}}- \hat{c}_\eps \|u\|_{1,\mathcal{H}}^{r_2} & \text{if } \|u\|_{1,\mathcal{H}}> 1.
	\end{cases}
    \end{align*}
    Since $p<q<r_1,r_2$ it follows that for some $\eta_0 \in (0,1)$ small enough
    \begin{align*}
	\ph(u) \geq \hat{\gamma}>0 \quad\text{for all }u \in \Wp{\mathcal{H}} \text{ with }\|u\|_{1,\mathcal{H}}=\eta_0.
    \end{align*}
    Now let $u \in N$ and take $s_u>0$ such that $s_u \|u\|_{1,\mathcal{H}}=\eta_0$. From Proposition \ref{proposition_nodal_2} we obtain
    \begin{align*}
	0<\hat{\gamma} \leq \ph(s_u u) \leq \ph(u) \quad\text{for all }u \in N,
    \end{align*}
    so $m>0$.
\end{proof}

As a direct consequence of Proposition \ref{proposition_positive_infimum_1} we obtain that $m_0>0$.

\begin{proposition}
    Let hypotheses \eqref{condition_poincare} and (H) be satisfied. Then $m_0>0$.
\end{proposition}

\begin{proof}
    Applying Proposition \ref{proposition_positive_infimum_1} and recall that $u^+, -u^- \in N$, we have for  each $u\in N_0$
    \begin{align*}
	\ph(u) = \ph(u^+)+\ph(-u^-) \geq 2m >0.
    \end{align*}
    Hence, $m_0>0$.
\end{proof}

\begin{proposition}\label{proposition_nodal_3}
    Let hypotheses \eqref{condition_poincare} and (H) be satisfied. Then there exists $y_0 \in N_0$ such that $\ph(y_0)=m_0$.
\end{proposition}

\begin{proof}
    Let $\{y_n\}_{n \geq 1} \subseteq N_0$ be a minimizing sequence, that is,
    \begin{align*}
	\ph(y_n) \searrow m_0.
    \end{align*}
    Clearly,
    \begin{align*}
	\ph(y_n)=\ph(y_n^+)+\ph(-y_n^-)
    \end{align*}
    with $y_n^+, -y_n^- \in N$. Similar to the proof of Proposition \ref{proposition_coercive} we can show that the sequences $\{y_n^+\}_{n \geq 1}, \{y_n^-\}_{n \geq 1}  \subseteq \Wp{\mathcal{H}}$ are bounded. Therefore, we may assume that
    \begin{align}\label{n16}
	\begin{split}
	    &y_n^+ \weak v_1 \quad\text{in } \Wp{\mathcal{H}},\quad v_1 \geq 0,\\
	    &y_n^- \weak v_2 \quad\text{in } \Wp{\mathcal{H}},\quad v_2 \geq 0.
	\end{split}
    \end{align}
    Suppose that $v_1=0$. Then, since $y_n^+ \in N$, it holds
    \begin{align*}
	0 = \l\lan \ph'(y_n^+),y_n^+\r\ran
	= \hat{\rho}_{\mathcal{H}}(y_n^+)
	-\into f(x,y_n^+)y_n^+\,dx-\int_{\partial\Omega} g(x,y_n^+)y_n^+\,d\sigma
    \end{align*}
    for all $n \in \N$. From \eqref{n16} and Proposition \ref{proposition_modular_properties} we conclude that
    \begin{align*}
	y_n^+ \to 0 \quad\text{in }\Wp{\mathcal{H}}.
    \end{align*}
    Hence
    \begin{align*}
	0<m\leq \ph(y_n^+) \to \ph(0)=0 \quad\text{as }n\to +\infty,
    \end{align*}
    which is a contradiction. Thus, $v_1\neq 0$. In a similar way we can show that $v_2\neq 0$. Taking Proposition \ref{proposition_nodal_1} into account there exists $t_1,t_2>0$ such that
    \begin{align*}
	t_1 v_1 \in N \quad\text{and}\quad t_2 v_2 \in N.
    \end{align*}
    Setting $y_0=t_1v_1-t_2v_2=y_0^+-y_0^-$ gives $y_0\in N_0$. Applying the sequentially weakly lower semicontinuity of $\ph$, Proposition \ref{proposition_nodal_2} and the fact that $y_0 \in N_0$ we obtain
    \begin{align*}
	m_0
	&= \lim_{n\to+\infty} \ph(y_n)\\
	& = \lim_{n\to+\infty} \left(\ph(y_n^+)+\ph(-y_n^-)\right)\\
	& \geq \liminf_{n\to+\infty} \left(\ph(t_1y_n^+)+\ph(-t_2y_n^-)\right)\\
	& \geq \ph(t_1v_1)+\ph(-t_2 v_2)\\
	& \geq \ph(y_0)\\
	& \geq m_0.
    \end{align*}
    Therefore
    \begin{align*}
	\ph(y_0)=m_0
    \end{align*}
    with $y_0 \in N_0$.
\end{proof}

\begin{proposition}\label{proposition_critical_point}
    Let hypotheses \eqref{condition_poincare} and (H) be satisfied. 
    Let $y_0 \in N_0$ be such that $\ph(y_0)=m_0$. Then $y_0 \in K_\ph$. 
    In particular $y_0\in \Wp{\mathcal{H}}\cap \Linf$ is a solution of problem \eqref{problem}.
\end{proposition}

\begin{proof}
    The proof of this proposition follows the idea of the proof of Theorem 1.4 in Liu-Dai \cite{Liu-Dai-2018} and exploits the quantitative deformation lemma of Willem, see Jabri \cite[Theorem 4.2]{Jabri-2003}.

    From hypothesis (H)(v), Proposition \ref{proposition_nodal_2} and the definition of $N_0$, for $s,t>0$ such that at least one of $s,t\ne 1$, we have
    \begin{align}\label{eq_l_s1}
	\ph\l(sy_0^+-ty_0^-\r)=\ph\l(sy_0^+\r)+\ph\l(-ty_0^-\r)<\ph\l(y_0^+\r)+\ph\l(-y_0^-\r)=\ph\l(y_0\r)=m_0.
    \end{align}

    Now we proceed by contradiction. So suppose that $\ph'(y_0)\ne 0$.  Then there exist $\delta>0$ and $\rho>0$ such that
    \begin{align*}
	\l\|\ph'(v)\r\|_{1,\mathcal{H}}\ge\rho\quad
	\text{for all } v\in W^{1,\mathcal{H}}(\Omega) \text{ with } \ \|v-y_0\|_{1,\mathcal{H}}\le 3\delta.
    \end{align*}
    Let
    \begin{align*}
	D=\l[\frac{1}{2},\frac{3}{2}\r]\times \l[\frac{1}{2},\frac{3}{2}\r].
    \end{align*}
    From \eqref{eq_l_s1}, we see that
    \begin{align*}
	\ph\l(sy_0^+-ty_0^-\r)=m_0\quad\text{if and only if}\quad s=t=1.
    \end{align*}
    Thus
    \begin{align*}
	\beta=\max_{(s,t)\in\partial D}\ph\l(sy_0^+ - t y_0^-\r)<m_0.
    \end{align*}
    Let
    \begin{align*}
	\eps=\min\l\{\frac{m_0-\beta}{4},\frac{\rho\delta}{8}\r\}.
    \end{align*}
    By the quantitative deformation lemma of Willem, see Jabri \cite[Theorem 4.2]{Jabri-2003}, there exists a continuous deformation $\eta\colon [0,1]\times W^{1,\mathcal{H}}(\Omega)\to W^{1,\mathcal{H}}(\Omega)$ such that
    \begin{enumerate}
	\item[(i)]
	    $\eta(1,v)=v$ if $v\not\in \ph^{-1}([m_0-2\eps,m_0+2\eps])$;
	\item[(ii)]
	    $\ph(\eta(1,v))\le m_0-\eps$ for all $v\in W^{1,\mathcal{H}}(\Omega)$ with $\|v-y_0\|_{1,\mathcal{H}}\le\delta$ and $\ph(v)\le m_0+\eps$;
	\item[(iii)]
	    $\ph(\eta(1,v))\le \ph(v)$ for all $v\in W^{1,\mathcal{H}}(\Omega)$.
    \end{enumerate}

    It follows easily that
    \begin{align}\label{eq_l_s2}
	\max_{(s,t)\in D}\ph\l(\eta(1,sy_0^+- ty_0^-)\r)<m_0.
    \end{align}
    Let us now define $h\colon\R_+\times\R_+\to W^{1,\mathcal{H}}(\Omega)$ by
    \begin{align*}
	h(s,t)=\eta\l(1,sy_0^+-ty_0^-\r)
    \end{align*}
    and put
    \begin{align*}
	H_0(s,t) & =  \Big(\langle\ph'(sy_0^+),y_0^+\rangle,\ \langle\ph'(-ty_0^-),-y_0^-\rangle\Big),\\
	H_1(s,t) & =  \left(\frac{1}{s}\l\langle\ph'(h^+(s,t)),h^+(s,t)\r\rangle,\ \frac{1}{t}\l\langle\ph'(-h^-(s,t)),-h^-(s,t)\r\rangle\r).
    \end{align*}
    Note that $\deg(H_0,D,0)=1$, as
    \begin{align*}
	\l\langle\ph'(sy_0^+),y_0^+\r\rangle &>0 \quad \text{and}\quad \l\langle\ph'(-sy_0^-),-y_0^-\r\rangle>0\quad \text{for all }s\in (0,1),\\
	\l\langle\ph'(sy_0^+),y_0^+\r\rangle & <0\quad \text{and}\quad \l\langle\ph'(-sy_0^-),-y_0^-\r\rangle<0\quad\text{for all }s>1.
    \end{align*}
    By \eqref{eq_l_s2} and property (i) of $\eta$ (see the choice of $\eps>0$), we have that
    \begin{align*}
	h(s,t)=s y_0^+-ty_0^-\quad\text{for all }(s,t)\in\partial D.
    \end{align*}
    Thus $H_0=H_1$ on $\partial D$ and hence
    \begin{align*}
	\deg(H_1,D,0)=\deg(H_0,D,0)=1.
    \end{align*}
    By the existence property of the Brouwer degree (see, for example, Gasi\'nski-Papageorgiou \cite[Theorem 4.11]{Gasinski-Papageorgiou-2016} or Papageorgiou-Winkert \cite[Theorem 6.2.22]{Papageorgiou-Winkert-2018}), we get
    \begin{align*}
	H_1(s,t)=0\quad\textrm{for some}\ (s,t)\in D.
    \end{align*}
    This means that
    \begin{align*}
	\eta\l(1,sy_0^+-ty_0^-\r)=h(s,t)\in N_0\quad\textrm{for some}\ (s,t)\in D.
    \end{align*}
    But this contradicts \eqref{eq_l_s2} and the definition of $m_0$.

    So, we conclude that $y_0\in K_{\ph}$ and thus $y_0$ is a solution of problem \eqref{problem}. From Proposition \ref{theorem_apriori} we have that $y_0 \in \Linf$.
\end{proof}

\begin{proposition}\label{proposition_two_nadal_domain}
    Let hypotheses \eqref{condition_poincare} and (H) be satisfied. If $y_0\in N_0$, is as in Proposition \ref{proposition_critical_point}, then $y_0$ is a nodal solution of problem \eqref{problem} which has exactly two nodal domains.
\end{proposition}

\begin{proof}
    From the definition of $N_0$ and Proposition \ref{proposition_critical_point} it is clear that $y_0\in N_0$ is a sign changing solution. It remains to show that $y_0$ has exactly two nodal domains. Arguing by contradiction, suppose that there exist disjoint open sets $\Omega_1$, $\Omega_2$ and $\Omega_3$ on which $y_0$ has fixed sign. Without any loss of generality, we may assume that $y_0$ has only three nodal domains. Let
    \begin{align*}
	y_k(x)=
	\begin{cases}
	    y_0(x) & \text{if }  x\in\Omega_k,\\
	    0      & \text{if }  x\in\Omega\setminus\Omega_k
	\end{cases}
    \end{align*}
    for $k=1,2,3$, $x\in\Omega$. Without any loss of generality, we may assume that
    \begin{align*}
	y_1\big |_{\Omega_1}>0,\quad y_2\big|_{\Omega_2}<0,\quad y_3\big|_{\Omega_3}<0.
    \end{align*}
    Setting $\hat{y}=y_1+y_2$, we have $\hat{y}^+=y_1$ and $\hat{y}^-=-y_2$. Since $y_0=y_1+y_2+y_3=\hat{y}+y_3$ and $\ph'(y_0)=0$ because of Proposition \ref{proposition_critical_point} we have
    \begin{align*}
	0=\l\lan \ph'(y_0),\hat{y}^+\r\ran
	=\l\lan \ph'(\hat{y})+\ph'(y_3),\hat{y}^+\r\ran
	=\l\lan \ph'(\hat{y}),\hat{y}^+\r\ran.
    \end{align*}
    Therefore $\l\lan \ph'(\hat{y}),\hat{y}^+\r\ran=0$. In the same way we can show that $\l\lan \ph'(\hat{y}),\hat{y}^-\r\ran=0$. From this we see that $\hat{y}^+,-\hat{y}^- \in N$ and so $\hat{y}\in N_0$.

    Applying Proposition \ref{proposition_nodal_3} and hypothesis (H)(iv) gives
    \begin{align*}
	m_0
	& = \ph(y_0)
	= \ph(y_0)-\frac{1}{q}\big\langle \ph'(y_0),y_0\big\rangle\\
	& = \ph(\hat{y})+\ph(y_3) -\frac{1}{q}\Big(\big\langle \ph'(\hat{y}),\hat{y}\big\rangle+\big\langle \ph'(y_3),y_3\big\rangle\Big)\\
	& = \ph(\hat{y})+\ph(y_3)-\frac{1}{q}\big\langle \ph'(y_3),y_3\big\rangle\\
	& = \ph(\hat{y})+\l(\frac{1}{p}-\frac{1}{q}\r)\|\nabla y_3\|_p^p+\l(\frac{1}{p}-\frac{1}{q}\r)\|y_3\|_p^p\\
	& \quad +\into \l(\frac{1}{q}f(x,y_3)y_3-F(x,y_3)\r)\,dx  +\int_{\partial\Omega} \l(\frac{1}{q}g(x,y_3)y_3-G(x,y_3)\r)\,d\sigma\\
	& \geq m_0+\l(\frac{1}{p}-\frac{1}{q}\r)\|\nabla y_3\|_p^p+\l(\frac{1}{p}-\frac{1}{q}\r)\|y_3\|_p^p.
    \end{align*}
    Since $p>q$, we see that $\Omega_3=\emptyset$. Thus we conclude that $y_0$ has only two nodal domains.
\end{proof}

Finally we can state the following multiplicity theorem for problem \eqref{problem} summarizing the results from Propositions \ref{proposition_constant_sign_solutions} and \ref{proposition_two_nadal_domain}.

\begin{theorem}\label{theorem_main_result}
    Let hypotheses \eqref{condition_poincare} and (H) be satisfied. Then, problem \eqref{problem} has at least three nontrivial solutions $u_0,v_0,y_0\in W^{1,\mathcal{H}}(\Omega)\cap \Linf$ such that
    \begin{align*}
	u_0 \geq 0,\quad v_0\leq 0,\quad y_0 \ \text{is nodal with two nodal domains}.
    \end{align*}
\end{theorem}

\section*{Acknowledgment}

The authors wish to thank the knowledgeable referee for his/her remarks in order to improve the paper as well as for the indication of further possible investigations.

\end{document}